\documentclass[11pt]{amsart}
\usepackage[bottom=1in, right=1in, left=1in, top=1in]{geometry}

\usepackage[dvipsnames,table]{xcolor}
\usepackage{appendix}

\RequirePackage[OT1]{fontenc}
\RequirePackage{amsthm,amsmath}
\RequirePackage[colorlinks,citecolor=magenta,urlcolor=blue,linkcolor=magenta]{hyperref}
\RequirePackage{comment}
\usepackage{caption}
\usepackage{graphicx}
\usepackage{subfigure}
\RequirePackage{latexsym,amssymb}

\RequirePackage{float,epsfig,multirow,rotating,times}
\RequirePackage{upgreek,wrapfig}
\usepackage{tikz}
\usepackage{float}
\usetikzlibrary{matrix}
\usetikzlibrary{arrows}
\usetikzlibrary{decorations.pathreplacing,calligraphy}
\usepackage{nicematrix}
\usepackage{blkarray}
\usepackage{mathdots}
\usepackage{multicol}
\usepackage{bbm}
\usetikzlibrary{patterns}
\usepackage{booktabs}

\usepackage{indentfirst}
\usepackage{textcomp}
\usepackage{lmodern}
\usepackage{mathrsfs}
\usepackage{url}
\usepackage{relsize}
\usepackage{tcolorbox}

\usepackage{ marvosym }
\usepackage[english]{babel}
\usepackage{makecell}
\usepackage{mathtools}
\usepackage{verbatim}
\usepackage[latin1]{inputenc}

\newtheorem{theorem}{Theorem}[section]
\newtheorem{lemma}[theorem]{Lemma}
\newtheorem{corollary}[theorem]{Corollary}

\newtheorem{proposition}[theorem]{Proposition}

\theoremstyle{definition}
\newtheorem{example}[theorem]{Example}
\newtheorem{remark}[theorem]{Remark}
\newtheorem{definition}[theorem]{Definition}

\numberwithin{subcase}{case}

\renewcommand{\tilde}[1]{\widetilde{#1}}

\newcommand{\RR}{{\mathbb R}}

\newcommand*{\QEDA}{\hfill\ensuremath{\Diamond}}

\begin{document}
\title{On the dimensions of correlated equilibrium polytopes of generic games}
\author{Jan Draisma}\address{Mathematical Institute, Sidlerstrasse 5,
3012 Bern, Switzerland}
\email{jan.draisma@unibe.ch} 
\author{Linda Hoyer}\address{RWTH Aachen University}
\email{linda.hoyer@rwth-aachen.de} 
\author{Irem Portakal}\address{Max Planck Institute for Mathematics in the Sciences}
\email{mail@irem-portakal.de} 

\thanks{JD was partially supported by project grant 200021-227864 from the Swiss National Science Foundation. LH was supported by the grant SFB-TRR 195 Symbolic Tools in Mathematics and their Application from the DFG, German Research Foundation.}

\begin{abstract}
In this paper, we study the dimension of the correlated equilibrium polytope of finite games. Under the oriented-matroid notion of genericity, we prove that if a generic game is not full-dimensional, then there exists a subgame whose correlated equilibrium polytope is affinely isomorphic to that of the original game. This settles and generalizes an earlier conjecture of Brandenburg, Hollering, and Portakal (2024). Moreover, we show that the existence of a correlated equilibrium whose slices are all non-zero implies that the correlated equilibrium polytope is either full-dimensional or a singleton.
\end{abstract}
\maketitle
\section{Introduction}
The notion of correlated equilibrium was introduced by Aumann in \cite{Aumann1974Subjectivity}. It generalizes the notion of Nash equilibrium \cite{nash1950equilibrium} by allowing each player to choose their strategy based on a private observation of a common public signal drawn from a joint probability distribution. If no player can increase their expected payoff by deviating from their prescribed strategy, assuming that the others adhere to theirs, then the joint probability distribution of the signal is called a correlated equilibrium. The set of correlated equilibria forms a convex polytope inside the probability simplex \cite{Aumann87}. In contrast to Nash equilibria, whose computation is PPAD-complete \cite{chen2006settling, daskalakis2009complexity} (and for which computing a second equilibrium is even NP-complete \cite[Corollary 2]{conitzer2008new}), correlated equilibria can be computed via linear programming and have therefore attracted considerable attention. Moreover, under certain necessary and sufficient linear conditions, the Nash equilibria lie on the relative boundary of the correlated equilibrium polytope \cite[Section 6.7]{Viossat2005}.

In this paper, for generic games, i.e.\ for a dense open set of games, our purpose is to study the dimensions of these polytopes. The openness criterion we use was defined in \cite[Section 5]{CombCorrEquilibria} via oriented matroid strata (Definition~\ref{de:Generic}). This open set was described in detail for $2\times n$ games. The authors in \cite[Conjecture 5.1]{combinatorics-correlated-equilibria} conjecture that whenever the correlated equilibrium polytope $P_G$ of a $2\times n$ game $G$ is not full-dimensional, there exists a generic $2\times\tilde{n}$ subgame, with $\tilde{n}<n$, whose correlated equilibrium polytope is full-dimensional and combinatorially equivalent to $P_G$. We establish a generalization of this conjecture for arbitrary finite games.

\begin{theorem}[{Theorem~\ref{MainTheorem}}]
Let $G$ be a generic game. If the correlated equilibrium polytope $P_G$ is not full-dimensional, then there is
a proper subgame $\tilde{G}$ of $G$ i.e.\ a game with smaller strategy sets, such that $P_{\tilde{G}}$ is full-dimensional and $P_G$ is affinely isomorphic to $P_{\tilde{G}}$.
\end{theorem}

A related line of work is the dual reduction for correlated equilibria introduced by Myerson \cite{MYERSON1997183}. Any correlated equilibrium polytope can be reduced by iterating this procedure to a full-dimensional one; however, while every correlated equilibrium of the reduced game is also a correlated equilibrium of the original game, Myerson shows in \cite[Section 5]{MYERSON1997183} that the converse does not always hold. In contrast, our reduction preserves the entire correlated equilibrium polytope up to affine isomorphism, and is therefore strictly stronger. 

In \cite[Proposition 1]{VIOSSAT20081152}, it was shown that the class of games with a unique correlated equilibrium forms an open set. Moreover, Viossat \cite{viossat2003elementary} studies games with full-dimensional correlated equilibrium polytopes using an auxiliary game. In Appendix B, he shows that the correlated equilibrium polytope $P_G$ is full-dimensional if and only if there exists a totally mixed correlated equilibrium $p \in \Delta^{\circ}_{D-1}$, and there exists a vector $x \in\mathbb{R}^D$ satisfying all nonvacuous incentive constraints with strict inequality. In contrast, we provide a sufficient condition under which $P_G$ is either full-dimensional or a singleton that only requires the existence of a correlated equilibrium $p\in P_G$ whose every slice $p_{-i,k}$ is nonzero for all $i\in[n]$ and $k\in[d_i]$.

\begin{theorem}[{Theorem~\ref{TheoremFullOrSingleton}}]
Let $G=(n,S,X)$ be a generic game. Assume there exists a point $p \in P_G$ such that its $k$-th slice in the $i$-th direction, $p_{-i,k}$, is nonzero for all $i \in [n]$ and $k \in [d_i]$. Then the correlated equilibrium polytope $P_G$ is either full-dimensional or a singleton.     
\end{theorem}

During the course of this paper, many examples were computed with the \texttt{Sagemath} script in \cite{SageCode} and GameTheory.m2 \cite{connelly2025gametheory} package.
\section{Preliminaries}
Let $n \geq 2$ be the number of players. For each $i \in
[n]:=\{1,\ldots,n\}$, let
$S^{(i)}$ be the set of pure strategies of player $i$. Moreover, let
$X^{(i)}_{j_1, \dots, j_n}$, $j_{i'} \in S^{(i')}$, be the payoff for
player $i$ when each player $i'$ chooses the pure strategy $j_{i'}$.
Set $X:=(X^{(1)}, \dots, X^{(n)})$, $S:=\prod_{i=1}^n S^{(i)}$, and
let $G:=(n,S,X)$ be the corresponding game in normal-form. 

\begin{definition}
Let $G=(n,S,X)$ be a game. A \textit{subgame}
$\tilde{G}=(n,\tilde{S},\tilde{X})$ of $G$ is a game where \[
\tilde{S}=\prod_{i=1}^n \tilde{S}^{(i)}, \ \tilde{S}^{(i)}
\subseteq S^{(i)} \ \text{for all} \ i \in [n] \] and
\[\tilde{X}=(\tilde{X}^{(1)}, \dots, \tilde{X}^{(n)}), \
\tilde{X}^{(i)}_{j_1, \dots, j_n}=X^{(i)}_{j_1, \dots, j_n} \text{ for
all } j_i \in \tilde{S}^{(i)}, \ i \in [n].\]
\end{definition}

Writing $d_i:=|S^{(i)}|$, we will say in the sequel that $G$ is a $(d_1
\times  \dots \times d_n)$-game. Set $D:=\prod_{i=1}^n d_i$. By slight
abuse of notation, we will identify $\RR^D$ with $\bigotimes_{i=1}^n
\RR^{S^{(i)}}$. By $\Delta_{D-1} \subseteq \RR^D$ we denote the
probability simplex. When no confusion is possible, we will just take
$S^{(i)}=[d_i]$. Of course, if we then pass to a subgame
with $|\tilde{S}^{(i)}|=\tilde{d}_i \leq d_i$, we can relabel the strategies in
the subgame to $1,\ldots,\tilde{d}_i$.

An element $p \in \Delta_{D-1}$ of the probability simplex determines
the distribution of a random variable taking values in $\prod_{i=1}^n [d_i]$. We then write
\[p_{-i,k}:=(p_{j_1, \dots ,j_{i-1},k,j_{i+1}, \dots, j_n})_{j_1 \in
[d_1], \dots, j_{i-1} \in [d_{i-1}], j_{i+1} \in [d_{i+1}], \dots, j_{n}
\in [d_{n}] } \in \bigotimes_{i' \neq i} \RR^{d_{i'}}=:\RR^{D/d_i}\]
for any $i \in [n], k \in [d_i]$. This is the $k$-th
slice of the tensor $p$ in the $i$-th direction, and if $p_{-i,k}$ is
not identically zero, then
up to scaling with the marginal, it represents the conditional probability distribution on
$\prod_{i' \neq i} [d_{i'}]$ given that the $i$-th entry of the random
variable equals $k$.

For each player $i$ and each $k,\ell \in [d_i]$, we define the linear
form $H^{X^{(i)}}_{k,\ell}$ on $\RR^{D/d_i}$ by
\[H^{X^{(i)}}_{k,\ell}(q):=\sum_{j \in \prod_{i' \neq i}
[d_{i'}]} \left(X^{(i)}_{j_1,\dots, j_{i-1},k,j_{i+1} \dots,
j_n}-X^{(i)}_{j_1,\dots, j_{i-1},\ell,j_{i+1} \dots, j_n}   \right)
q_{j_1,\dots, j_{i-1},k,j_{i+1} \dots, j_n};   \] 
note that $H^{X^{(i)}}_{\ell,k}=-H^{X^{(i)}}_{k,\ell}$.

\begin{definition} \label{de:CorrEq}
We say that $p \in \Delta_{D -1}$ is a \textit{correlated equilibrium} if 
\begin{equation} H^{X^{(i)}}_{k,\ell}(p_{-i,k}) \geq 0 \label{eq: incentive constraints}\end{equation}
for all players $i \in [n]$ and all pure strategies $k, \ell \in [d_i]$. We denote $P_G \subseteq \Delta_{D-1}$ to be the set of all correlated equilibria.

A tuple $(x^{(1)}, \dots, x^{(n)}) \in \Delta_{d_1-1} \times \dots \times
\Delta_{d_n-1}$ is a \textit{Nash equilibrium} if $x^{(1)} \otimes \dots
\otimes x^{(n)} \in P_G$, and this Nash equilibrium is called {\em totally
mixed} if for all $i$, all entries of $x^{(i)}$ are (strictly)
positive. \end{definition}

The inequalities \eqref{eq: incentive constraints} in the definition are called the {\em incentive constraints} and
have the following interpretation. If an independent party draws $j \in
\prod_{i=1}^n [d_i]$ from the probability distribution $p\in\Delta_{D-1}$ and proposes
the pure strategy $j_i$ to each player $i$, then in the conditional
distribution obtained by scaling $p_{-i,j_i}$ appropriately, player $i$'s
expected payoff is at least as large when playing pure strategy $j_i$ as when playing any other pure strategy. Note that $P_G \subseteq \Delta_{D-1}$ is a convex
polytope, first studied in \cite[Section 4g]{Aumann87}, and that it is non-empty since it contains the set of Nash equilibria \cite{nash1950equilibrium}.

\begin{remark} \label{re:IncConsEq}
For future use, we make the following observation: if $p \in P_G$, $i \in
[n]$, and $k,\ell \in [d_i]$ are such that the slices $p_{-i,k}$ and $p_{-i,\ell}$
are both non-zero and scalar multiples of each other, then in fact the
incentive constraints $H^{X^{(i)}}_{k,\ell}(p_{-i,k}) \geq 0$ and
$H^{X^{(i)}}_{\ell,k}(p_{-i,\ell}) \geq 0$ both hold with equality.
\end{remark}

\begin{definition} \label{de:Generic}
Let $A_G$ be the coefficient matrix of the linear system of inequalities defining the cone spanned by $P_G \subseteq \RR^D$ consisting of both the incentive constraints \eqref{eq: incentive constraints} and the inequalities expressing that the entries of $p$
are non-negative. We say that the game $G=(n,S,X)$ is \textit{generic}, if it is generic
in the sense of \cite[Construction 5.3]{CombCorrEquilibria}, i.e.\ if all non identically zero $D \times D$-subdeterminants (maximal minors) of $A_G$ are non-zero for $X$. For any two games $G=(n,S,X)$ and $G'=(n,S,X')$ we write $G \sim G'$ if the sign tuple of those determinants ($0$, $+$, or $-$) are precisely
the same for $G'$ as for $G$.
\end{definition}

The collection of all regions where the maximal minors exhibit a given sign pattern is called the \textit{oriented matroid stratification}. Our notion of genericity above relies on this notion, which determines the combinatorial type of the polytope $P_G$, itself fully specified by the underlying oriented matroid. Note that genericity is an open condition, and if $G=(n,S,X)$ is
generic, then any $G'=(n,S,X')$ with $X'$ sufficiently close to $X$
satisfies $G \sim G'$. Throughout the paper, unless it is specified, a generic game means generic with respect to Definition~\ref{de:Generic}.

\begin{example}[Generic $2\times2$ games]
Let $G$ be a $2\times 2$ game. The coefficient matrix $A_G$ from Definition~\ref{de:Generic} has $8$ rows and $4$ columns. The first four rows come from the $\mathcal{H}$-representation of $P_G$, and the remaining four rows encode the non-negativity of $p_{ij}$ for $i,j \in [2]$.
\[
A_G=\begin{pmatrix}
X^{(1)}_{11} - X^{(1)}_{21} & X^{(1)}_{12} - X^{(1)}_{22} & 0 & 0 \\
0 & 0 & X^{(1)}_{21} - X^{(1)}_{11} & X^{(1)}_{22} - X^{(1)}_{12} \\
X^{(2)}_{11} - X^{(2)}_{12} & 0 & X^{(2)}_{21} - X^{(2)}_{22} & 0 \\
0 & X^{(2)}_{12} - X^{(2)}_{11} & 0 & X^{(2)}_{22} - X^{(2)}_{21} \\
1 & 0 & 0 & 0 \\
0 & 1 & 0 & 0 \\
0 & 0 & 1 & 0 \\
0 & 0 & 0 & 1
\end{pmatrix}
\]
There are in total $\binom{8}{4}=70$ maximal minors of this matrix, out of which $45$ are not identically zero. The algebraic boundary of the oriented matroid strata is given by the union of the varieties defined by these $45$ ($4 \times 4$)-minors. It consists of four linear irreducible components, each corresponding to the vanishing of one of the unknown entries of $A_G$:
\[
\bigl(
X^{(2)}_{21} - X^{(2)}_{22},\,
X^{(1)}_{12} - X^{(1)}_{22},\,
X^{(1)}_{11} - X^{(1)}_{21},\,
X^{(2)}_{11} - X^{(2)}_{12}
\bigr).
\]
In the complement of this algebraic boundary, we obtain the classes of generic games. By enumerating all possible sign patterns in this complement using \texttt{Mathematica}~\cite{combinatorics-correlated-equilibria}, we recover the result of~\cite{calvo2003set} that there are two types of generic $2 \times 2$ games: their correlated equilibrium set is either a single point or a $3$-dimensional polytope that is a bipyramid over a triangle. For example, the Bach or Stravinsky and Hawk or Dove games share the same sign (positive) pattern, and their correlated equilibrium polytopes are full-dimensional. In contrast, the Prisoners' Dilemma and Matching Pennies games have correlated equilibrium polytopes is a single point. In this smallest example, one can also see that the open regions of same sign pattern of the oriented matroid need not be connected. \QEDA
\end{example}

\begin{remark} Let $G$, $G'$ be generic games such that $G \sim G'$.
\label{RemarkGenericness}
\begin{enumerate}
    \item The face posets of $P_G$ and $P_{G'}$ are the same. Indeed, it
    is well known that the face poset of a polyhedral cone is determined
    by the oriented matroid defined by the rows of its coefficient matrix,
    and the face posets of the polytopes $P_G$ and $P_{G'}$ are readily
    obtained from those of the cones spanned by them.

    \item Let $i \in [n], k\in [d_i]$. If there exists a $p \in P_G$
    such that $p_{-i,k} \neq 0$, then there exists a $q \in P_{G'}$
    such that $q_{-i}(k) \neq 0$. This follows from the (proof of the)
    previous item.

    \item Let $\tilde{S}^{(i)} \subseteq S^{(i)}$ be a subset of pure
    strategies, $\tilde{S}:=\prod_{i=1}^n \tilde{S}^{(i)}$, and let $\tilde{G}$ (resp. $\tilde{G'}$) be the corresponding
    subgame of $G$ (resp. $G'$). Then $\tilde{G}$ and $\tilde{G'}$
    are generic games and $\tilde{G} \sim \tilde{G'}$. Indeed, set
    $\tilde{D}:=|\tilde{S}|$ and fix any $\tilde{D}
    \times \tilde{D}$-submatrix $\tilde{\delta}$ 
    of the coefficient matrix $A_{\tilde{G}}$. For each 
    $j \in S \setminus \tilde{S}$ add to $\delta$ the column corresponding to the
    variable $p_j$ and the row corresponding to the inequality $p_j
    \geq 0$. This yields a $D \times D$-submatrix of $A_G$ that, up to row and
    column permutations, equals
    \[ \delta=\begin{pmatrix} \tilde{\delta} & 0 \\ 0 & I_{D-\tilde{D}}
    \end{pmatrix}. \]
    Now the determinant of $\tilde{\delta}$ is (up to the sign
    corresponding to those permutations) equal to that of $\delta$.
    This implies both statements. 

\end{enumerate}    
\end{remark}

\noindent We can now formulate the main theorem. For a more precise statement, see Theorem~\ref{thm:MainTheorem2}.

\begin{theorem}
\label{MainTheorem}
Let $G$ be a generic game. If $P_G$ is not full-dimensional, then there is
a proper subgame $\tilde{G}$ of $G$ such that $P_{\tilde{G}}$ is full-dimensional
and the polytope $P_G$ is affinely isomorphic to $P_{\tilde{G}}$.
\end{theorem}

\section{Proof of the main theorem}
Let $n$ be a positive integer and let $d_i \geq 2$ be integers for any
$i \in [n]$. Let $G=(n,S,X)$ be a $(d_1\times \dots \times
d_n)$-game. 

\begin{definition}
Let $p \in \Delta_{D -1}$. For any $i \in [n]$ and $k \in [d_i]$,
define the polytope \[ \mathcal{C}^{X^{(i)}}_k:=  \{ q \in
\Delta_{D/d_i-1} \mid  H^{X^{(i)}}_{k,\ell}(q)  \geq 0 \ \text{for all} \ \ell \in [d_i] \} \subseteq \Delta_{D/d_i-1}. \]   Moreover, we define the projection map \[ \pi_{-i,k}: \Delta_{D -1}  \to  \Delta_{D/d_i-1} \cup \{0\}, \ p \mapsto \begin{cases}
    0 & \ \text{if} \ p_{-i,k}=0, \\
     \frac{p_{-i,k}}{\mathbf{1}^t p_{-i,k} } & \ \text{else}, 
\end{cases}   \] where $\mathbf{1} \in \RR^{D/d_i}$ is the vector where all entries are equal to $1$. 
\end{definition}

Note that $\pi_{-i,k}$ sends a probability distribution $p$ to (zero
or) the conditional probability distribution represented by the $k$-th
slice of $p$ in the $i$-th direction; and that if $p \in P_G$, then
$\pi_{-i,k}(p) \in \mathcal{C}^{X^{(i)}}_k \cup \{0\}.$

\begin{remark}
Let $i \in [n]$ and $k \in [d_i]$. The polytope $\mathcal{C}^{X^{(i)}}_k$ has the following interpretation:

\noindent For any $j \in [n] \backslash \{i\}$, assume that player $j$ chooses a mixed strategy $x^{(j)} \in \Delta_{d_j-1}$. Then \[ x^{(1)} \otimes \dots 
\otimes x^{(i-1)} \otimes x^{(i+1)} \otimes \dots \otimes x^{(n)} \in \mathcal{C}^{X^{(i)}}_k \] if and only if the choice of the $k$-th pure strategy for player $i$ maximizes their own payoff. As such, the polytope $\mathcal{C}^{X^{(i)}}_k$ is also known as a \textit{best-response region}, which appears in the Lemke-Howson algorithm for computing Nash equilibria of bimatrix games; see also \cite{FindingNashEquilibria}.
\end{remark}

\begin{definition}
We say that $p \in P_G$ is \textit{balanced} (for $G$) if and only if
for all $i \in [n]$ and all $k, \ell \in [d_i]$ the implication 
\[ \pi_{-i,k}(p) \in \mathcal{C}^{X^{(i)}}_{\ell} \Rightarrow \pi_{-i,k}(p)=\pi_{-i,\ell}(p) \in \mathcal{C}^{X^{(i)}}_{k} \cap \mathcal{C}^{X^{(i)}}_{\ell}  \]     
holds. Note that the opposite implication always holds by
Remark~\ref{re:IncConsEq}.
\end{definition}

For instance, if $(x^{(1)},\dots,x^{(n)})$ is a totally mixed Nash
equilibrium for $G$, then $p:=x^{(1)} \otimes \dots \otimes x^{(n)} \in
P_G$ is balanced, since for every $i$, all slices of $p$ in the $i$-th
direction are non-zero scalar multiples of $x^{(1)} \otimes \dots \otimes
x^{(i-1)} \otimes x^{(i+1)} \otimes \dots \otimes x^{(n)}$ (see Remark~\ref{re:IncConsEq}).
\begin{example}\label{ex: 3x3 game balanced point}
Consider the following generic $(3 \times 3)$-game $G$ with the payoff matrices \[ (X^{(1)},X^{(2)})= \begin{pmatrix}
    (10,12) & (15,10) & (11,13) \\
    (5,9) & (4,15) & (19,8) \\
    (18,18) & (9,17) & (13,10)
\end{pmatrix}.  \]
The correlated equilibrium polytope $P_G \subset \Delta_8$ has the f-vector $(79, 373, 774, 924, 698, 338, 100, 16)$ and in particular it is full-dimensional. Its unique vertex $(\frac{7}{10}, \frac{17}{70}, \frac{2}{35}) \otimes (\frac{13}{109},\frac{37}{109}, \frac{59}{109})$ in the interior of $\Delta_8$ is the unique totally mixed Nash equilibrium of the game and in particular it is balanced. Now consider the following point which lies in the interior of an edge of $P_G$, \[p:=\frac{1}{200}\begin{pmatrix}
   0 & 56 & 77 \\
   0 & 24 & 33 \\
   10 & 0 & 0 
\end{pmatrix} = \frac{19}{20} \begin{pmatrix}
   0 & \frac{28}{95} & \frac{77}{190} \\
   0 & \frac{12}{95} & \frac{33}{190} \\
   0 & 0 & 0 
\end{pmatrix} + \frac{1}{20} \begin{pmatrix}
   0 & 0 & 0 \\
   0 & 0 & 0 \\
   1 & 0 & 0 
\end{pmatrix} \in P_G.  \]
We obtain that $p_{-1,k} \in \mathcal{C}^{X^{(1)}}_{\ell}$ for $k,\ell \in \{1,2\}$ and $p_{-2,k} \in \mathcal{C}^{X^{(1)}}_{\ell}$ for $k \in \{2,3\}$, and these are the only non-trivial statements of this form. Moreover, we have that $\pi_{-1,1}(p) = \pi_{-1,2}(p)$ and $\pi_{-2,2}(p) = \pi_{-2,3}(p)$ respectively, showing that $p$ is balanced. 
\QEDA
\end{example}

The following statement shows that we can always slightly perturb the payoff tensors to make any given point in the correlated equilibrium polytope $P_G$ balanced.
\begin{proposition}
\label{PropositionBalancing}
Let $G=(n,S,X)$ and let $p \in P_G$. Then there
exist $Y$ arbitrarily close to $X$ such that the game $G':=(n,S,Y)$
satisfies $p \in P_{G'}$ and such that $p$ is balanced for $G'$.
\end{proposition}

\begin{proof}
Fix $i \in [n]$. We are going to perturb $X^{(i)}$; we can do this
independently for all $i$ to obtain the required tuple $Y$.  Consider the
finite set
\[ W:=\{\pi_{-i,k}(p) \mid p_{-i,k} \neq 0, k \in [d_i]\} \subseteq
\Delta_{D/d_i-1} \]
of normalized non-zero slices of $p$ in the $i$-th direction. Pick a
vertex $w_1 \in W$ of the convex hull of $W$, then a vertex $w_2 \in W
\setminus \{w_1\}$ of the convex hull of $W \setminus \{w_1\}$, and so
on; this gives an enumeration $W=\{w_1,\ldots,w_m\}$ with the property
that for each $j \in [m]$, there exists a linear function $h_j \in
(\RR^{D/d_i})^*$ with $h_j(w_j)>0$ and $h_j(w_{j'})<0$ for all $j'>j$.
Now initialize $Y^{(i)}:=X^{(i)}$; we will update $Y^{(i)}$ in $m$
steps.

In the first step, we choose a sufficiently small $\varepsilon_1 > 0$ 
and for $k \in [d_i]$ update the $k$-th slice in the $i$-th direction $Y^{(i)}_k$ via 
\[Y^{(i)}_k:=\begin{cases}
	Y^{(i)}_k + \varepsilon_1 h_1 & \text{if } \pi_{-i,k}(p)=w_1, \text{ and}\\
	Y^{(i)}_k  & \text{otherwise}. 
	\end{cases}
\]
Since $\varepsilon_1$ is sufficiently small, this preserves the incentive
constraints $H^{Y^{(i)}}_{k,\ell}(p_{-i,k}) \geq 0$ that hold with a strict
inequality. Assume that $k,\ell \in [d_i]$ are such that this constraint holds
with equality. Now:
\begin{itemize}
\item if $\pi_{-i,k}(p),\pi_{-i,\ell}(p) \neq w_1$,
then all terms in the constraint remain unchanged, so that it continues
to hold with equality;
\item if $\pi_{-i,k}(p)=\pi_{-i,\ell}(p)=w_1$, then the $+$ terms in the
incentive constraint increase by $\varepsilon_1 h_1(p_{-i,k})>0$ and the $-$
terms decrease by the same amount, so again the incentive constraint
continues to hold with equality;
\item if $\pi_{-i,k}(p)=w_1\neq
\pi_{-i,\ell}(p)$, then the $+$ terms increase by
$\varepsilon_1 h_1(p_{-i,k})>0$ while the $-$ terms remain unchanged, so the
incentive constraint holds with a strict inequality after the update;
and 
\item if $\pi_{-i,k}(p) \neq w_1=\pi_{-i,\ell}(p)$, then the
$+$ terms remain the same, while in the $-$ terms we additionally
subtract $\varepsilon_1 h_1(p_{-i,k}) \leq 0$, so the contraint remains
valid. Moreover, if $p_{-i,k} \neq 0$,
then in fact $\pi_{-i,k}(p)=w_j$ for some $j>1$, and hence
$h_1(p_{-i,k})<0$, so the constraint holds with a strict inequality
after the update. 
\end{itemize}
Summarising, after this update all incentive constraints remain valid,
and in fact we achieve that $H^{Y^{(i)}}_{k,\ell}(p_{-i,k}) > 0$ if $p_{-i,k}$
is non-zero and moreover precisely one of $k,\ell$ has the property that
the corresponding slice is a non-zero scalar multiple of $w_1$.

In the remaining $m-1$ steps, we proceed in exactly the same manner,
but now with $w_j$ and $h_j$ for $j=2,\ldots,m$, taking care to
choose the new $\varepsilon_j$ small enough so that strictly satisfied
incentive constraints remain strictly satisfied. Afterwards, we
then have $H^{(i)}_{k,\ell}(p_{-i,k})>0$ for all $k,\ell \in [d_i]$
with $p_{-i,k} \neq 0$ and $\pi_{-i,k}(p) \neq \pi_{-i,\ell}(p)$
(this also holds when $\pi_{-i,\ell}(p)=0$). Since
$H^{(i)}_{k,\ell}=-H^{(i)}_{\ell,k}$, this
implies that $\pi_{-i,k}(p) \not \in \mathcal{C}^{X^{(i)}}_\ell$ for
such $k,\ell$. Doing this for all $i$, we find $Y$, arbitrarily near $X$,
such that $p$ is balanced for the game $(n,S,Y)$.
\end{proof}

\begin{example}
    Consider the game $G$ of Bach or Stravinsky with the payoff matrices \[
(X^{(1)},X^{(2)})=
\begin{pmatrix}
(3,2) & (0,0) \\
(0,0) & (2,3)
\end{pmatrix}.
\]
The correlated equilibrium polytope $P_G$ is three-dimensional, which is a bipyramid over a triangle. The vertex $p$ =  $\begin{pmatrix}
\frac{2}{7} & \frac{3}{7} \\
0 & \frac{2}{7} 
\end{pmatrix} \in  P_G \subsetneq \Delta_3$ is not balanced, since e.g.\ $\pi_{-1,1} \in \mathcal{C}_2^{X^{(1)}}$ but $\pi_{-1,1}(p) \neq \pi_{-1,2}(p)$. By the proof of Proposition~\ref{PropositionBalancing}, we can obtain the following perturbed game $G'$ such that $G \sim G'$.
\[
(Y^{(1)},Y^{(2)})=
\begin{pmatrix}
\left(5,\,3\right)
&
\left(-1,\,0\right)
\\
\left(0,\,-{2}\right)
&
\left(2,\,3\right)
\end{pmatrix}.
\]
In particular, $P_G$ and $P_{G'}$ have the same face poset, and $p \in P_{G'}$ is balanced for $G'$.
\QEDA
\end{example}

\noindent In the proof of Proposition~\ref{PropositionBalancing}, we have grouped together the slices of $p$ in
the $i$-th direction that are positive scalar multiples of the same
$w_j$. We will use this partition more extensively below and introduce
the following notation.

\begin{definition} \label{de:Partition}
Let $p \in \Delta_{D-1}$. For each $i \in [n]$, we define an
equivalence relation on $[d_i]$ where $k \sim \ell$ if and only if
$\pi_{-i,k}(p)=\pi_{-i,\ell}(p)$. We let \[M^{(i)}_1 \dot{\cup}
\dots \dot{\cup} M^{(i)}_{m_i} = [d_i]\] be the associated set
partition and set $d^{(i)}_a:=|M^{(i)}_a|$ for $a \in [m_i]$.

\noindent We define 
\begin{equation} \label{eq:rp} r(p):=\sum_{i=1}^n |\{ \pi_{-i,k}(p)
\mid p_{-i,k} \neq 0, k \in [d_i]   \}|; \end{equation}
note that the $i$-th term here is precisely the cardinality of the set
$W$ in the proof of Proposition~\ref{PropositionBalancing}.
\end{definition}
\begin{remark}
\label{RemarkR}
Let $p \in \Delta_{D-1}$.
\begin{enumerate}
    \item It holds that \[r(p) = \left( \sum_{i=1}^n m_i \right) - |\{i \in [n] \mid p_{-i,k}=0 \ \text{for some} \ k \in [d_i]   \}  |.  \]
    \item It holds that $n \leq r(p) \leq \sum_{i=1}^n d_i$, with
    $r(p)=n$ if and only if $p=(x^{(1)} \otimes \dots \otimes
    x^{(n)})$, where $(x^{(1)}, \dots, x^{(n)})$ is a Nash equilibrium
    of $G$. Indeed, in the latter case, all non-zero slices of $p$ in
    the $i$-th direction are positive scalar multiples of $x_1 \otimes
    \cdots \otimes x_{i-1} \otimes x_{i+1} \otimes \cdots \otimes x_n$, so each $i \in [n]$
    contributes $1$ to the sum in \eqref{eq:rp}. Conversely, if each $i$
    contributes $1$, then it follows that in each direction, the
    non-zero slices of $p$ are positive scalar multiples of each other,
    so that $p$ has rank one as claimed. 
    \item Let $q \in \Delta_{D-1}$ and let $p':=(1- \varepsilon)p +
    \varepsilon q \in \Delta_{D-1}$ for some $0< \varepsilon <1$.
    Then for $\varepsilon$ small enough, if the $k$-th and $\ell$-th
    slice of $p$ in the $i$-th direction are linearly independent, then 
    the same holds for $p'$, hence $r(p') \geq r(p)$. 
    \item Let $a_1 \in [m_1], \dots, a_n \in [m_n]$ be indices and
    consider \[ \tilde{p}:=(p_{j_1, \dots, j_n})_{j_i \in
    M^{(i)}_{a_i}, i \in [n]}.\] Then $\tilde{p}=0$ if 
    for at least one $i \in [n]$, we have $\pi_{-i,k}(p)=0$ for some,
    and hence all, $k \in M_{a_i}^{(i)}$. If that is not the case,
    then in each direction all slices of $\tilde{p}$ 
    are positive scalar multiples of each other. This implies
    that $\tilde{p}=y^{(1)} \otimes \cdots \otimes y^{(n)}$ for
    certain vectors $y^{(i)} \in \RR^{M^{(i)}_{a_i}}$ with 
    positive entries. 
\end{enumerate}
\end{remark}

\begin{lemma}
\label{LemmaSubgameAdd}
Let $p \in P_G$ be balanced. Let $a_1 \in [m_1], \dots, a_n \in [m_n]$ such that $p_{-i,k} \neq 0$ for any $i \in [n]$, $k \in M^{(i)}_{a_i}$ and
let $\tilde{G}=(n,\tilde{S},\tilde{X})$ be the subgame of $G$
corresponding to the subsets $M_{a_i}^{(i)} \subseteq S^{(i)}$. Let
$\tilde{q} \in P_{\tilde{G}}$ be a correlated equilibrium, and extend
$\tilde{q}$ to a tensor $q \in \Delta_{D-1}$ by assigning $0$ to any
tuple of pure strategies not in $\tilde{S}$. Then there exists $0<\varepsilon<1$
such that $(1-\varepsilon)p + \varepsilon q \in P_G$.
\end{lemma}
\begin{proof}
Let $i \in [n]$, $k \in M_{a_i}^{(i)}$ and $\ell \in [d_i]$. We show that there exists an $0<\varepsilon^{(i)}_{k,\ell} <1$ such that \[
    H^{X^{(i)}}_{k,\ell} \left(
    (1-\varepsilon^{(i)}_{k,\ell})p_{-i,k} +
    \varepsilon^{(i)}_{k,\ell} q_{-i,k}  \right) \geq 0. 
\]

\noindent First assume that $\ell \in M_{a_i}^{(i)}$. Now, since $\tilde{q} \in
P_{\tilde{G}}$, it follows that $H^{X^{(i)}}_{k,\ell}(q_{-i,k}) \geq 0$. Moreover, $p \in P_G$, so also $H^{X^{(i)}}_{k,\ell}(p) \geq 0$ and we are done in this case.

So assume now that $\ell \not\in M_{a_i}^{(i)}$. Since $p$ is balanced
and $p_{-i,k} \neq 0$, we have $H^{X^{(i)}}_{k,\ell}(p_{-i,k}) > 0$. By then choosing $\varepsilon^{(i)}_{k,\ell}$ small enough, we can thus guarantee that  \[ H^{X^{(i)}}_{k,\ell} \left( (1-\varepsilon^{(i)}_{k,\ell})p + \varepsilon^{(i)}_{k,\ell} q  \right) \geq 0. \]

\noindent Finally, we set \[ \varepsilon:= \mathrm{min} \{ \varepsilon^{(i)}_{k,\ell} \mid i \in [n], k \in M_{a_i}^{(i)}, \ell \in [d_i]   \}\] and we are done.
\end{proof}

\begin{corollary}
\label{CorollaryNonZero}
Assume there exists a $p \in P_G$ such that $p_{-i,k} \neq 0$ for any
$i \in [n], k \in [d_i]$. Then there exist $X'$ arbitrarily close to $X$
such that the game $G':=(n,S,X')$ has a $p' \in P_{G'}$ all of whose
entries are positive and that satisfies $r(p') \geq r(p)$.
\end{corollary}
\begin{proof}
We use induction over the number of zero entries of $p$; let \[z(p):=\{(j_1, \dots, j_n) \mid j_i \in [d_i], p_{j_1, \dots, j_n}=0  \}.\]

Clearly, if $z(p)=0$, i.e., all entries of $p$ are positive, we can just take $G'=G$, $p'=p$ and we are done.

So assume now that $z(p)>0$. By Proposition \ref{PropositionBalancing}, there 
are $X'$ arbitrarily close to $X$ such that $p \in P_{G'}$ is balanced
for the game $G':=(n,S,X')$. By relabeling the pure strategies as well
as the parts in the partitions of $[d_i]$ for all $i$, we may assume
that $p_{1, \dots, 1}=0$ and that $1 \in M^{(i)}_1$ for all $i \in
[n]$. Then the $M^{(1)}_1 \times \cdots \times M^{(n)}_1$-block in $p$ is identically zero.  Let $\tilde{G}=(n,\tilde{S},\tilde{X})$ be the subgame
corresponding to the subsets $M^{(i)}_1 \subseteq S^{(i)}$ and let
$\tilde{q} \in P_{\tilde{G}}$ be a correlated equilibrium, which we
extend to a tensor $q \in \Delta_{D-1}$ by assigning $0$ to any 
tuple of pure strategies not in $\tilde{S}$. Then by Lemma
\ref{LemmaSubgameAdd} and Remark \ref{RemarkR}(3), there is an
$0<\varepsilon <1$ such that $p':=(1- \varepsilon)p + \varepsilon q
\in P_{G'}$ and $r(p') \geq r(p)$. Further, $p$ has zeros on the
positions where $q$ is positive, so we find that 
$z(p')<z(p)$. We are done by the induction hypothesis. 
\end{proof}

\begin{proposition} \label{prop:NotAllNash}
    Let $G=(n,S,X)$ be a generic game. If $P_G$ is a not a singleton, then it contains points $p \in \Delta_{D-1}$ that are {\em not} of the form $x_1 \otimes \cdots \otimes x_n$ for $(x_1,\ldots,x_n)$ a Nash equilibrium of $G$.
\end{proposition}

Loosely speaking, if $G$ is generic and $P_G$ is not a singleton, then it does not consist of Nash equilibria alone. This is certainly true for almost all games, since these have finitely many Nash equilibria \cite{Harsanyi73, Wilson71}. But this needs a rigorous proof since the notions of genericity are defined differently. On the other hand, the Nash equilibria are on the topological boundary of $P_G$ by \cite{NGH03}, hence if $P_G$ is full-dimensional, then it does not only consist of Nash equilibria either. The point of the proposition above is that it applies to generic games in our sense, and even if $P_G$ is lower-dimensional. 

\begin{proof}
    Assume that $P_G$ is not a singleton and does consist of Nash equilibria only. For distinct $p=x^{(1)} \otimes \cdots \otimes x^{(n)}, q=y^{(1)} \otimes \cdots \otimes y^{(n)} \in P_G$, where $x^{(i)},y^{(i)} \in \Delta_{d_i-1}$, the convex combination $(p+q)/2$ is not a rank-one tensor unless $x^{(i)}=y^{(i)}$ for all but one $i \in [n]$, say for all $i \in [n-1]$. Let $r=z^{(1)}\otimes \cdots \otimes z^{(n)} \in P_G$, and assume that $z^{(i)} \neq x^{(i)} (=y^{(i)})$ for some $i \in [n-1]$. Since also $(r+p)/2$ and $(r+q)/2$ are in $P_G$, $z^{(n)}$ must equal both $x^{(n)}$ and $y^{(n)}$, a contradiction. Hence we find that 
    \[ P_G=\{x^{(1)} \otimes \cdots \otimes x^{(n-1)} \otimes u \mid u \in Q\} \]
    where $Q \subseteq \Delta_{d_n-1}$ is a polytope.

    Let $Z_G$ be the set of rows of the matrix $A_G$ from Definition~\ref{de:Generic} corresponding to inequalities that hold with equality on $P_G$. Let $\tilde{S}^{(i)} \subseteq [d_i]$ be the set of slices that are not identically zero on $P$; for $i \leq n-1$ this is the support of $x^{(i)}$ and for $i=n$ it is the set of coordinates that are not identically zero on $Q$. Now $Z_G$ contains the rows corresponding to the inequalities $p_{j_1,\ldots,j_n} \geq 0$ for $(j_1,\ldots,j_n) \not \in \tilde{S}^{(1)} \times \cdots \times \tilde{S}^{(n)}$, and---because $P_G$ consists of Nash equilibria only---also the rows corresponding to the incentive constraints $H_\ell^{X^{(i)}}(p_{-i,k}) \geq 0$ for all $i$ and all $k,\ell \in \tilde{S}^{(i)}$. (This last observation is also used in the proof of \cite[Proposition 2]{NGH03}.)

    Now let $G'=(S,n,Y)$ be a any game with $Y$ sufficiently close to $X$. We claim that $P_{G'}$ also consists of Nash equilibria only. Suppose, for a contradiction, that $p' \in P_{G'}$ does not correspond to a Nash equilibrium; we may assume that $p'$ is in the relative interior of $P_{G'}$. Since $G$ is generic, $A_{G'}$ defines the same oriented matroid as $A_G$, and hence $Z_{G'}=Z_G$. In particular, $p'$ has support $\tilde{S}^{(1)} \times \cdots \times \tilde{S}^{(n)}$ and also satisfies the incentive constraints $H^{Y^{(i)}}_\ell(p'_{-i,k}) \geq 0$ for all $i$ and $k,\ell \in \tilde{S}^{(i)}$. Since $p'$ is not a Nash equilibrium, however, there exist $i \in [n]$ and $k,\ell \in \tilde{S}^{(i)}$ such that the $k$-th and $\ell$-th slice of $p'$ are not scalar multiples of each other. Since we do have $p'_{-i,k} \in \mathcal{C}^{Y^{(i)}}_\ell$ (and the same statement with $k,\ell$ reversed), $p'$ is not balanced for $G'$. By Proposition~\ref{PropositionBalancing}, there exists $X''$ aribitrarily near $Y$ so that $p'$ {\em is} balanced for $G'':=(S,n,X'')$. However, this contradicts the fact that $X''$ is still arbitrarily near $X$ and hence $A_{G''}$ still defines the same matroid. This proves the claim.

    So we have found that for $Y$ in a neighbourhood of $X$, $P_{G'}$ consists of (infinitely many) Nash equilibria only. But this contradicts the result from \cite{Harsanyi73} that games with finitely many Nash equilibria are dense.
\end{proof}

\begin{theorem}
\label{TheoremFullOrSingleton}
Let $G=(n,S,X)$ be a generic game. Assume there exists a $p \in P_G$ such that $p_{-i,k} \neq 0$ for any $i \in [n]$, $k \in [d_i]$. Then the correlated equilibrium polytope $P_G$ is either full-dimensional or a singleton.     
\end{theorem}
\begin{proof}
Assume $P_G$ is not a singleton. By Proposition~\ref{prop:NotAllNash}, we may choose $p \in P_G$ with $p_{-i,k} \neq 0$ for all $i \in [n]$, $k \in [d_i]$ that moreover is not a Nash equilibrium, so that $n<r(p)$. 
We argue by induction on the number \[
\left(\sum_{i=1}^n d_i\right)-r(p).\] 
If $r(p)=\sum_{i=1}^n d_i$, then all
incentive constraints for $p$ are satisfied with a strict inequality
and thus $P_G$ is full-dimensional. So assume from now on that that $r(p)<\sum_{i=1}^n d_i$. 

By Proposition \ref{PropositionBalancing} and Corollary
\ref{CorollaryNonZero}, the point $p$ can be assumed balanced and to
have only positive entries. 

We will now construct
a rank-1-tensor $q \in \Delta_{D-1}$, together with an $0<\varepsilon<1$
and an $X'$ arbitrarily close to $X$ such that for the game
$G':=(n,S,X')$ we have $p':=(1-\varepsilon)p
+ \varepsilon q \in P_{G'}$ and $r(p')>r(p).$ By
the induction hypothesis, the theorem then follows.

To find $q$, we begin by relabeling the players, their pure strategies, and the
parts in the partitions of the $[d_i]$ conveniently, as follows:
\begin{enumerate}
    \item Assume that $M^{(i)}_1=[n_1^{(i)}]$ for all $i \in [n]$. 
    \item Since $r(p)< \sum_{i=1}^n d_i$, there is at least one $i$
    for which the partition of $[d_i]$ from
    Definition~\ref{de:Partition} has a part that is not a
    singleton. So we may assume that $n_1^{(1)} \geq 2$.
    \item The condition that $r(p)>n$ implies that, as a tensor, $p$
    has rank $>1$. Then it has some flattening to a matrix that also
    has rank $>1$. This implies that there are at least two indices
    $i$ with $m_{i}>1$, say $i_1$ and $i_2$. 
\end{enumerate}

Let $\tilde{G}=(n,\tilde{S},\tilde{X})$ be the subgame with
$\tilde{S}:=\prod_{i=1}^n M^{(i)}_1$. Let $(y^{(1)},\dots, y^{(n)})$
be a Nash equilibrium of $\tilde{G}$ and let $\tilde{q}:=y^{(1)}
\otimes \dots \otimes y^{(n)} \in P_{\tilde{G}}$, which we extend to a
tensor $q \in \Delta_{D-1}$ by assigning $0$ to any tuple of pure
strategies not in $\tilde{S}$. 

For each $i \in [n]$, the slices $p_{-i,k}$ with $k \in [d^{(i)}_1]$
are all non-zero scalar multiples of each other. Hence 
there is a unique $x^{(i)} \in \Delta_{d^{(i)}_1 -1}$ such that
\[p_{-i,k}= x^{(i)}_k  \sum_{\ell =1}^{d^{(i)}_1} p_{-i,\ell}\]
for those $k$. We define the tensor $\tilde{t} := x^{(1)} \otimes \dots \otimes x^{(n)}$.

First assume that $\tilde{q} \neq \tilde{t}$. Then there is an index $i$
such that $x^{(i)} \neq y^{(i)}$. Since $x^{(i)},y^{(i)}$ are both in the
probability simplex $\Delta_{d^{(i)}-1}$, this implies that
$x^{(i)},y^{(i)}$ are linearly independent. Let
$A$ be the $d^{(i)}_1 \times \left(\prod_{i' \neq i} d_{i'}\right)$-matrix whose
rows are the slices of $p$ labeled by $k \in [d^{(i)}_1]$ (vectorized in
some fixed but arbitrary manner); so $A$ is part of a
flattening of $p$. Let $B$ be the corresponding part of a flattening
of $q$. Then $A$ and $B$ are rank-one matrices with column spaces spanned
by $x^{(i)}$ and $y^{(i)}$, respectively, hence different column
spaces since $x^{(i)},y^{(i)}$ are linearly independent.

Now at least one of $i_1,i_2$ from item (3) above is not equal to $i$;
say that $i_2 \neq i$. This implies that the rows of $A$ have positive
entries in positions whose $i_2$-th coordinate is $>d^{(i_2)}_1$,
whereas $B$ only has zero entries there. This means that also the column
spaces of $A$ and $B$ are distinct.  Then for any $0 < \varepsilon <1$
the combination $C:=(1-\varepsilon)A + \varepsilon B$ is a rank-two matrix. 
Moreover, by Lemma \ref{LemmaSubgameAdd}, we may choose this
$\varepsilon$ such that $p':=(1-\varepsilon)p+\varepsilon q$ lies in $P_G$.
We conclude that the slices of $p'$ in the $i$-th direction labeled by
$k \in [d^{(i)}_1]$, which correspond to the rows of $C$, are not all
scalar multiples of each other. Now, Remark \ref{RemarkR}(3) implies that (by potentially decreasing $\varepsilon$ further) $r(p')>r(p)$. 

We are left with the case that $\tilde{q} = \tilde{t}$. Recall that $d^{(1)}_1 \geq 2$; we will adjust the payoff tensor $X^{(1)}$. 
At least one of $i_1$ and $i_2$ is not equal to $1$, say $i_2 \neq 1$. Now 
$p_{-1,1}$ has positive entries on positions whose $i_2$-coordinate
is $>d^{(i_2)}_1$, while $q_{-1,1}$ has zeros there. Hence the 
slices $p_{-1,1}$ and $q_{-1,1}$ are linearly independent, and there
exists a linear function $h \in (\RR^{D/d_i})^*$ with 
$h(p_{-1,1})=0>h(q_{-1,1})$. Define the
$(d_1 \times \dots \times d_n)$-tensor $Y^{(1)}$ by setting its $k$-th slice
in the first direction to be 
\[Y^{(1)}_k= \begin{cases}
    X^{(1)}_1 + \varepsilon h & \text{if } k=1, \text{ and}\\
    X^{(1)}_k & \text{otherwise.}
\end{cases} 
\] for some small $\varepsilon>0$. Now set $X':=(Y^{(1)},X^{(2)}, \dots,
X^{(n)})$ and $G':=(n,S,X')$ with the subgame $\tilde{G'}$ corresponding
to the subsets strategies $M^{(1)}_1, \dots, M^{(n)}_1$. We claim that, if
$\varepsilon$ is sufficiently small, then $p \in P_{G'}$ and $p$ is balanced
for $G'$. Indeed, the only incentive constraints that may have changed
compared to those for $G$ are the inequalities
$H^{Y^{(1)}}_{k,\ell}(p_{-1,k})
\geq 0$ when at least one of $k,\ell$ is equal to $1$. However:
\begin{itemize} 
\item if $k=1$, then since $h(p_{-1,1})=0$,
the constraint in fact does not change; 
\item if $\ell=1$ and $k \in [d^{(1)}_1] \setminus \{1\}$, then $p_{-1,k}$ is a scalar multiple of
$p_{-1,1}$ and hence $h(p_{-1,k})=0$, and again the constraint does
not change; and 
\item if $\ell=1$ and $k>d^{(1)}_k$, then since $p$ is balanced we have
a strict inequality $H^{X^{(1)}}_{k,\ell}(p_{-1,k})>0$, hence
choosing $\varepsilon$
sufficiently small we still have $H^{Y^{(1)}}_{k,\ell}(p_{-1,k})>0$. 
\end{itemize}
The balancedness of $p$ for $G'$ follows because we keep all the
strict inequalities among the incentive constraints. 

Now the first two slices of $q$ in the first direction are positive
scalar multiples of each other; this follows from 
$d^{(1)}_1>1$, the fact that $q$
has rank $1$, and the fact that $y^{(1)}=x^{(1)}$, where the latter vector
only has positive entries. We therefore have
$H^{X^{(1)}}_{1,2}(q_{-1,1})=0$ by Remark~\ref{re:IncConsEq}. Hence
\[ H^{Y^{(1)}}_{1,2}(q_{-1,1})=H^{X^{(1)}}_{1,2}(q_{-1,1})+\varepsilon h(q_{-1,1})= \varepsilon h(q_{-1,1})<0.\] Thus $\tilde{q} \not\in P_{\tilde{G'}}$, i.e., $(x^{(1)},\dots, x^{(n)})$ is not a Nash equilibrium of $\tilde{G'}$. Finally, let $(z^{(1)}, \dots, z^{(n)})$ be a Nash equilibrium of $\tilde{G'}$ and let $\tilde{u}:=z^{(1)} \otimes \dots \otimes z^{(n)}$. Then $\tilde{u} \neq \tilde{t}$, so we have reduced to the former case and we are done.
\end{proof}

The following theorem now gives a precise version of Theorem \ref{MainTheorem}.
\begin{theorem}
\label{thm:MainTheorem2}
Let $G$ be a generic game.  For each $i \in [n]$, let \[ S_c^{(i)}:=
\{k \mid k \in [d_i], \ p_{-i,k} \neq 0 \ \text{for some} \ p \in P_G \}
\] be the set of pure strategies that have a positive marginal probability
for some correlated equilibrium. Let $\tilde{G}$ be the subgame of $G$
with pure strategies $\tilde{S}=\prod_{i=1}^n S_c^{(i)}$. Then $P_G$
is affinely isomorphic to $P_{\tilde{G}}$ and $P_{\tilde{G}}$ is either
full-dimensional or a singleton.
\end{theorem}
\begin{proof}
By assigning $0$ to the pure strategies in $S \backslash \tilde{S}$,
we can think of $P_{\tilde{G}}$ as a polytope in $\Delta_{D-1}$. The
choice of $\tilde{S}$ implies that $P_G \subseteq P_{\tilde{G}}$: the
incentive constraints for $G$ and the zeros on positions in $S \setminus
\tilde{S}$ imply the (fewer) incentive constraints for $\tilde{G}$. By
Theorem \ref{TheoremFullOrSingleton} applied to $\tilde{G}$ (which is
generic by Remark~\ref{RemarkGenericness}(3)), the latter polytope is
either full-dimensional or a singleton. So the theorem follows if we can
show that the the inclusion $P_G \subseteq P_{\tilde{G}}$ is an equality.

If $P_{\tilde{G}}$ is a singleton, then since $P_G$ is non-emtpy, we are done. So assume now that $P_{\tilde{G}}$ is not a singleton and that there is a $p \in P_{\tilde{G}}$ such that $p \not\in P_G$; we wish to derive a contradiction. 

Since $p \not\in P_G$, there are $i \in [n]$ and indices $m, \ell \in
[d_i]$ with $m \in S^{(i)}_c$, $\ell \not\in S^{(i)}_c$, such that
$H^{X^{(i)}}_{m,\ell}(p_{-i,m})<0.$ Since $p \in P_{\tilde{G}}$, in
this latter inequality we have $\geq 0$ whenever $\ell \in S^{(i)}_c$. These two
observations imply that the maximal inner product of a slice of $X^{(i)}$
with $p_{-i,m}$ is achieved at some slice whose label $j$ is not in
$S^{(i)}_c$. For this $j$, we have $H^{X^{(i)}}_{j,k}(p_{-i,m})
\geq 0$ for all $k \in [d_i]$. Note that slightly increasing all
entries of $X^{(i)}_j$ does not influence the subgame $\tilde{G}$ at
all. Furthermore, since $G$ is generic, this also does not affect the
support sets $S^{(i)}$ and the fact that $p \not \in P_{G}$. So we can
assume without loss of generality that $H^{X^{(i)}}_{j,k}(p_{-i,m})
> 0$ for any $k \in [d_i]$.

Let $q \in P_G$ be a correlated equilibrium such that $q_{-i,k} \neq 0$
for any $k \in S^{(i)}_c$. By Proposition \ref{PropositionBalancing},
there are $Y$ arbitrarily close to $X$ such that $q \in P_{G'}$ is
balanced for the game $G':=(n,S,Y)$. Clearly $Y$ can be chosen
such that $H^{Y^{(i)}}_{j,k}(p_{-i,m}) > 0$ for all $k \in [d_i]$.

Let $\tilde{G'}$ be the subgame of $G'$ corresponding to $\tilde{S}$.
By Theorem \ref{TheoremFullOrSingleton}, the correlated equilibrium
polytope $P_{\tilde{G'}}$ is full-dimensional.  We may choose a $t' \in
P_{\tilde{G'}}$ in the relative interior, i.e., all incentive constraints
coming from $\tilde{S}$ are positive. Since $q$ is balanced, there now is
an $0<\varepsilon <1$ such that $q':=(1-\varepsilon)q+ \varepsilon t'$
lies in $P_{G'}$,  is balanced, and has 
all incentive constraints coming from $\tilde{S}$ positive.

We now define the tensor $u \in \Delta_{D-1}$ by its slices in the
$i$-th direction: 
\[u_{-i,k} =  \begin{cases}
   \pi_{-i,m} (p) & \ \text{if} \ k = j, \\
   0 & \ \text{if} \ k \neq j.
\end{cases}   \] 

Choose an $0<\varepsilon' <1$ and set $q'':=(1-\varepsilon')q'+
\varepsilon' u$. We now show that for $\varepsilon'$ small enough, we
have $q'' \in P_{G'}$. By construction, the incentive
constraints for player $i$ are non-negative. So consider $i' \in [n]
\backslash \{i \}$. We have $q''_{-i',k} =0$ for any pure strategy $k \not\in S^{(i')}_c$, so assume now that $k \in S^{(i')}_c$ and let $\ell \in [d_{i'}] \backslash \{k\}$. But since $q'$ is balanced and all incentive constraints coming from $\tilde{S}$ are positive, we have that $H^{X^{(i')}}_{k,\ell}(q'_{-i',k})>0$. So indeed, by choosing $\varepsilon'$ small enough, it follows that  $H^{X^{(i')}}_{k,\ell}(q''_{-i',k})>0$. In conclusion, $q'' \in P_{G'}$ with $q''_{-i,j} \neq 0$. By Remark \ref{RemarkGenericness}, there now also exists a $v \in P_G$ with $v_{-i,j} \neq 0$, but this contradicts $j \not\in S^{(i)}_c$ and we are done.
\end{proof}

\begin{example}
We consider the generic $(3 \times 3)$-game $G$ from Example~\ref{ex: 3x3 game balanced point}:
 \[ (X^{(1)},X^{(2)})= \begin{pmatrix}
    (10,12) & (15,10) & (11,13) \\
    (5,9), & (4,15) & (19,8) \\
    (18,18) & (9,17) & (13,10)
\end{pmatrix}.  \]
Consider the balanced correlated equilibrium point \[ p:= \frac{1}{200} \begin{pmatrix}
   0 & 56 & 77 \\
   0 & 24 & 33 \\
   10 & 0 & 0 
\end{pmatrix} \in P_G. \]  Since the last
two columns of $p$ are linearly dependent, and so are the first two
rows of $p$, we have $r(p)=4>2$. Also, $p_{-i,k} \neq 0$ for $i \in \{1,2\}, k \in
\{1,2,3\}$, so by Theorem \ref{TheoremFullOrSingleton}, the correlated
equilibrium polytope $P_G$ is full-dimensional. We wish to confirm this
by constructing a game $G'=(2,S,Y)$ with $Y$ near $X$ and a $p' \in P_{G'}$ in
the relative interior.

The point $p$ leads to the set
partitions $M^{(1)}_1=\{1,2\}, M^{(1)}_2=\{3\}$ and $M^{(2)}_1=\{1\},
M^{(2)}_2=\{2,3\}$. We now use Lemma \ref{LemmaSubgameAdd}:   Consider the
$(1 \times 2)$-subgame corresponding to the sets $M^{(1)}_2,M^{(2)}_2$ with payoff matrices $\begin{pmatrix}
    (9,17) & (13,10)
\end{pmatrix}$ with a correlated equilibrium $\begin{pmatrix}
   1 & 0
\end{pmatrix}$ which is coming from a
Nash equilibrium. Analogously, we compute the Nash equilibrium
$\begin{pmatrix} 1 \\ 0 \end{pmatrix}$ of the
$(2 \times 1)$-subgame corresponding to $M^{(1)}_1,M^{(2)}_1$ with payoff matrices $\begin{pmatrix}
    (10,12) \\
    (5,9)
\end{pmatrix}$. We confirm that the
convex combination \[q:= \frac{1}{22} \begin{pmatrix}
    0 & 0 & 0 \\
    0 & 0 & 0 \\
    0 & 1 & 0
\end{pmatrix}+ \frac{1}{22} \begin{pmatrix}
    1 & 0 & 0 \\
    0 & 0 & 0 \\
    0 & 0 & 0
\end{pmatrix}+ \frac{20}{22} \ p = \frac{1}{220} \begin{pmatrix}
   10 & 56 & 77 \\
   0 & 24 & 33 \\
   10 & 10 & 0 
\end{pmatrix}  \] lies in $P_G$, in accordance with
Lemma~\ref{LemmaSubgameAdd}. In fact, we are not
allowed to apply that lemma twice without checking in between that the sum
is still balanced, which in this case it would not be, as the following
discussion shows.

Note that $r(q)=6$, which is the maximal possible value. However, $q$
is not balanced: We have that $\pi^{(1)}_2(q) \in
\mathcal{C}^{X^{(1)}}_1 \cap \ \mathcal{C}^{X^{(1)}}_2$ (respectively
$\pi^{(2)}_3(q) \in \mathcal{C}^{X^{(2)}}_2 \cap \
\mathcal{C}^{X^{(2)}}_3$), but $\pi^{(1)}_1(q) \not\in
\mathcal{C}^{X^{(1)}}_2$ (respectively $\pi^{(2)}_2(q) \not\in
\mathcal{C}^{X^{(2)}}_3$). We use the method presented in the proof of Proposition
\ref{PropositionBalancing} to balance $q$.  Our first goal is to modify the second row of $X^{(1)}$ to slightly enlarge $\mathcal{C}^{X^{(1)}}_2$ in the right direction, see also Figure \ref{FigBalancing}.
\begin{figure}[ht]
\centering
\resizebox{0.5\textwidth}{!}{%
\begin{tikzpicture}
\draw (0,0) -- (10,0) -- (0,10) -- (0,0);
\draw (1.5,3.0) -- (3,0);
\draw (1.5,3.0) -- (0,5.0);
\draw (1.5,3.0) -- (4.3,5.7);
\draw (-0.2,-0.4) node {\large$(0\ 0 \ 1)$};
\draw (10.2,-0.4) node {\large$(1 \ 0\ 0)$};
\draw (-0.2,10.4) node {\large$(0 \ 1 \ 0)$};
\draw (2,6) node {\Large$\mathcal{C}^{X^{(1)}}_1$};
\draw (1.7,0.9) node {\Large$\mathcal{C}^{X^{(1)}}_2$};
\draw (4.5,3) node {\Large$\mathcal{C}^{X^{(1)}}_3$};
\draw[red] (0,5.0) -- (5,5);
\draw[red] (1,4.5) -- (5,5);
\draw[red] (1,4.5) -- (0,5.0);
\fill[pattern color=red,pattern=north east lines]  (0,5.0) -- (5,5) -- (1,4.5) -- cycle;
\draw[red] (2.5,5.5) node {\Large$Q$};
\filldraw[black] (0,5.0) circle (4pt) node[anchor=south east]{\large$w=\pi^{(1)}_2(q)$};
\filldraw[black] (5,5) circle (4pt) node[anchor=south west]
{\large$\pi^{(1)}_{3}(q)$};
\filldraw[black] (1,4.5) circle (4pt) node[anchor=north west]
{\large$\pi^{(1)}_{1}(q)$};
\draw[loosely dashed, color = blue] (0.5,11) -- (0.5,-1);
\draw[blue] (0.8,-0.5) node {\Large$h$};

\end{tikzpicture}}
\caption{The subdivision $(\mathcal{C}^{X^{(1)}}_1,\mathcal{C}^{X^{(1)}}_2,\mathcal{C}^{X^{(1)}}_3)$ of $\Delta_2$ and the relative position of the projections $\pi^{(1)}_i(q)$, which are the vertices of the triangle $Q$. The hyperplane $h$ seperates the vertex $w$ from the other two, which we then use to modify the entries of $(X^{(1)}_{2,1},X^{(1)}_{2,2},X^{(1)}_{2,3} )$ to stabilize $q$.}
\label{FigBalancing}
\end{figure}
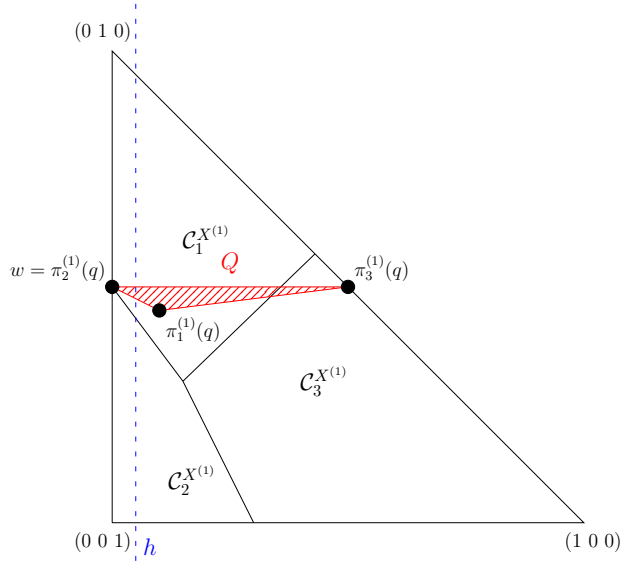

Consider the triangle $Q=
\mathrm{conv}(\{\pi^{(1)}_1(q),\pi^{(1)}_2(q),\pi^{(1)}_3(q)\})$ and
let $w:=\pi^{(1)}_2(q)$. We define the hyperplane linear function $h
\in (\RR^3)^*$ by 
\[h(t):=-15t_{1}+t_{2}+t_{3}.\] Then $h(q_{-1,2})>0$ and $h(q_{-1,1})<0$, $h(q_{-1,3})<0$. We add a small
multiple $\varepsilon h$ to the second row of $X^{(1)}$. In our case,
$\varepsilon=\frac{1}{100}$ does the trick and we obtain 
\[Y^{(1)}=\begin{pmatrix}
  10 & 15 & 11 \\
    4.85 & 4.01 & 19.01 \\
    18 & 9 & 13   
\end{pmatrix}.  \] Secondly, we do the same procedure for player $2$ (where we adjust the third column) analogously and arrive at the new payoff-matrix for player $2$ to be \[Y^{(2)}=\begin{pmatrix}
  12 & 10 & 13.01 \\
    9 & 15 & 8.01 \\
    18 & 17 & 9.9   
\end{pmatrix}.  \] Now set $Y:=(Y^{(1)},Y^{(2)})$ and $G':=(2,S,Y)$.
Then $G \sim G'$ and $q \in P_{G'}$ is balanced.

We are in a position to use Lemma~\ref{LemmaSubgameAdd} yet twice more
time to get rid of the remaining zero entries of $q$. So we consider
the two subgames with payoff-matrices $((4.85,9))$ and $((13,9.9))$
of $G'$, calculate the unique Nash equilibria  of these $1 \times 1$-subgames ($(1)$ in both cases) and define  \[ p':=\frac{1}{2202} \begin{pmatrix}
    0 & 0 & 0 \\
    1 & 0 & 0 \\
    0 & 0 & 0
\end{pmatrix}+ \frac{1}{2202} \begin{pmatrix}
    0 & 0 & 0 \\
    0 & 0 & 0 \\
    0 & 0 & 1
\end{pmatrix}+\frac{2200}{2202}q=\frac{1}{2202} \begin{pmatrix}
    100 & 560 & 770 \\
    1 & 240 & 330 \\
    100 & 100 & 1
\end{pmatrix} \in P_{G'}, \] where all the incentive constraints are positive. In conclusion, it holds that $P_{G'}$ (and therefore $P_G$) is full-dimensional. 
\end{example}

\bibliographystyle{abbrv}
\bibliography{bibliography}
\end{document}